\documentclass[12pt,a4paper]{amsart}
\usepackage{dsfont}
\usepackage{graphicx}
\usepackage{amssymb}
\usepackage[all]{xy}
\usepackage{amsmath}
\usepackage{tikz}
\usepackage[french,english]{babel}
\usepackage{amssymb,amsmath,amsfonts,amsthm,amscd}
\usepackage{indentfirst,graphicx}
\usepackage[font={scriptsize},captionskip=5pt]{subfig}
\usepackage{tikz}
\usetikzlibrary{matrix}
\usepackage[T1]{fontenc}
\usepackage[utf8]{inputenc}
\usepackage{geometry}

\newtheorem{theorem}{Theorem}[section]
\newtheorem{corollary}[theorem]{Corollary}
\newtheorem{proposition}[theorem]{Proposition}
\newtheorem{lemma}[theorem]{Lemma}

\newtheorem{definition}[theorem]{Definition}

\newtheorem{example}[theorem]{Example}

\newcommand{\cB}{{\mathcal B}}

\newcommand{\cI}{{\mathcal I}}
\newcommand{\cJ}{{\mathcal J}}

\newcommand{\cO}{{\mathcal O}}

\newcommand{\cX}{{\mathcal X}}

\newcommand{\cM}{{\mathcal M}}

\newcommand{\cOXp}{{\cO_X^p}}
\newcommand{\bC}{{\mathbb{C}}}

\newcommand{\bN}{{\mathbb{N}}}

\begin{document}
	
	\title{THE GENERIC EQUIVALENCE AMONG THE LIPSCHITZ SATURATIONS OF A SHEAF OF MODULES}
	\author{Terence Gaffney and Thiago da Silva}
	
	\maketitle
	
	\begin{abstract}
		{\small In this work, we extend the concept of the Lipschitz saturation of an ideal to the context of modules in some different ways, and we prove they are generically equivalent.}
	\end{abstract}
	
	\let\thefootnote\relax\footnote{2010 \textit{Mathematics Subjects Classification} 32S15, 14J17, 32S60
		
		}

	\section*{Introduction}
	
	The definition of the Lipschitz saturation of an ideal appears in \cite{G2}, in the context of bi-Lipschitz equisingularity. The study of bi-Lipschitz equisingularity was started by Zariski \cite{Z}, Pham and Teissier \cite{PT, P1}, and was further developed by Lipman \cite{L}, Mostowski \cite{M1, M2}, Parusinski \cite{PA1, PA2}, Birbrair \cite{B} and others.
	
	The Lipschitz saturation and the double of an ideal $I$, denoted $I_S$ and $I_D$, respectively, were defined in \cite{G2}, where $I$ is a sheaf of ideals of $\cO_X$, the analytic local ring of an analytic variety $X$. The ideal $I_S$ consists of elements in $\cO_X$ whose quotient of its pullback by the blowup-map, with a local generator of the pullback of $I$, is Lipschitz. The double $I_D$ is the submodule of $\cO_{X\times X}^2$ generated by $(h\circ\pi_1,h\circ\pi_2)$, $h\in I$, where $\pi_1,\pi_2:X\times X\rightarrow X$ are the projections. Theorem 2.3 of \cite{G2} gives an equivalence between $I_S$ and the integral closure of $I_D$, namely: $$h\in I_S\iff h_D\in\overline{I_D}.$$
	
	In \cite{SG}, the authors generalize the notion of the double for a sheaf of modules $\cM$. We will recall the main of these definitions and basic results about the double in Section 1. 
	
	As seen in \cite{G3}, we already know that for a submodule $M \subseteq \cO_{X,x}^p$ and $h \in \cO_{X,x}^p$, $$h\in\overline{M}\iff I_k(h,M)\subseteq\overline{I_k(M)}$$
	
	In Proposition \ref{P18}, we obtain a new criterion for the integral closure of $M$ using ideals, and all these mentioned equivalences inspire definitions of Lipschitz saturations by replacing the integral closure with the Lipschitz saturation of the ideal appearing on the right-hand side of the characteristic condition. 
	
	In Section 2, we will examine each of these definitions in detail, explore some basic properties, and discuss the relationships between them. Finally, in Section 3, we will demonstrate that these definitions coincide generically in a sheaf of submodules of $\cO_{X}^p$.

		
		
		
		

	\section*{Acknowledgements}
	
	The authors are grateful to Nivaldo Grulha Jr. for his careful reading of this work and to Maria Aparecida Soares Ruas for the valuable conversations about the subject of this work, especially in the generic equivalence among the Lipschitz saturations defined here.
	
	The first author was supported in part by PVE-CNPq, grant 401565/2014-9. The second author was supported by Funda\c{c}\~ao de Amparo \`a Pesquisa do Estado de S\~ao Paulo - FAPESP, Brazil, grant 2013/22411-2.

	\section{Background on Lipschitz saturation of ideals and integral closure of modules}\label{sec0}
	
	The Lipschitz saturation of a complex analytic local algebra was defined by Pham and  Teissier in \cite{PT}.
	
	\begin{definition}
		Let $I$ be an ideal of $\cO_{X,x}$,  $SB_I(X)$ the saturation of the blow-up and $\pi_S: SB_I(X)\rightarrow X$ the projection map. The {\bf Lipschitz saturation} of the ideal $I$ is denoted $I_S$, and is the ideal $$I_S:=\{h\in\cO_{X,x}\mid\pi_S^*(h)\in\pi_S^*(I)\}.$$
	\end{definition}
	
	Since the normalization of a local ring $A$ contains the Lipschitz Saturation of $A$, it follows that $I\subseteq I_S\subseteq\overline{I}$. In particular, if $I$ is integrally closed, then $I_S=\overline{I}$.
	
	This definition can be given an equivalent statement using the theory of integral closure of modules. Since Lipschitz conditions depend on controlling functions at two different points as the points come together, we should look for a sheaf defined on $X\times X$. We describe a way of moving from a sheaf of ideals on $X$ to a sheaf on $X\times X$. 
	
	Let $\pi_1,\pi_2: X\times X\rightarrow X$ be the projections to the i-th factor, and let $h\in\cO_{X,x}$. Define $h_D\in\cO_{X\times X,(x,x)}^{2}$ as $(h\circ\pi_1,h\circ\pi_2)$, called the double of $h$. We define the double of the ideal $I$, denoted $I_D$, as the submodule of $\cO_{X\times X,(x,x)}^2$ generated by $h_D$, where $h$ is an element of $I$.
	
	We can see in \cite{G2} the following result gives a link between Lipschitz saturation and integral closure of modules.
	
	\begin{theorem}[\cite{G2}, Theorem 2.3]
		Suppose $(X,x)$ is a complex analytic set germ, $I\subseteq\cO_{X,x}$ and $h\in\cO_{X,x}$.  
		\noindent Then \begin{center} $h\in I_S$ if, and only if, $h_D\in\overline{I_D}$.\end{center} 
	\end{theorem}
	
	Using the Lipschitz saturation of ideals (and doubles), in \cite{G1}, we defined the infinitesimal Lipschitz conditions for hypersurfaces.
	
	In \cite{G3}, there is a way to make a link between the integral closure of modules and ideals using minors of a matrix of generators of $\cM$. 
	
	Let $\cM$ be a submodule sheaf of $\cOXp$, and $[\cM]$ a matrix of generators of $\cM$. For each $k$, let $I_k(\cM)$ denote the ideal of $\cO_X$ generated by the $k\times k$ minors of $[\cM]$. If $h\in\cOXp$, let $(h,\cM)$ be the submodule generated by $h$ and $\cM$.
	
	\begin{proposition}[\cite{G3}, Corollary 1.8]\label{proposition G}
		Suppose $(X,x)$ is a complex analytic germ with irreducible components $\{V_i\}$. Then, $h\in\overline{\cM}$ at $x$ if, and only if, $I_{k_i}((h,\cM_i))\subseteq \overline {I_{k_i}(\cM_i)}$ at $x$, where $\cM_i$ is the submodule of $\cO_{V_i,x}^p$ induced from $\cM$ and $k_i$ is the generic rank of $(h,\cM_i)$ on $V_i$. 
	\end{proposition}
	
	Now we recall some basic results about the double of a module developed in \cite{SG}.
	
	Let $X \subseteq \mathbb{C}^n$ be an analytic space, and let $\cM$ be an $\cO_X $-submodule of $\cO_X^p$. Consider the projection maps $\pi_1,\pi_2: X \times X \rightarrow X$. We assume that $\cM$ is finitely generated by global sections.
	
	\begin{definition}
		Let $h\in\cOXp$. The double of $h$ is defined as the element
		
		$$h_D:=(h\circ\pi_1,h\circ\pi_2) \in\cO_{X\times X}^{2p}.$$
		
		The double of $\cM$ is denoted by $\cM_D$, and is defined as the $\cO_{X\times X}$-submodule of $\cO_{X\times X}^{2p}$ generated by $\{h_D \mid  h\in \cM\}$.
		
	\end{definition}
	
	Consider $z_1,\hdots,z_n$ the coordinates on $\mathbb{C}^n$. The next lemma is a useful tool to deal with the double.
	
	\begin{lemma}[\cite{SG}, Lemma 2.2]\label{L1} Suppose $\alpha\in\cO_X$ and $h\in\cOXp$. Then:
		
		\begin{enumerate}
			\item [a)] $(\alpha h)_D=-(0,(\alpha\circ\pi_1-\alpha\circ\pi_2)(h\circ\pi_2))+(\alpha\circ\pi_1)h_D$;
			
			\item [b)] $(0,(\alpha\circ\pi_1-\alpha\circ\pi_2)(h\circ\pi_2))\in \cM_D$, for all $h\in \cM$ and $\alpha\in\cO_X$;
			
			\item [c)] $\alpha\circ\pi_1-\alpha\circ\pi_2 \in I(\Delta(X))=(z_1\circ\pi_1-z_1\circ\pi_2,\ldots, z_n\circ\pi_1-z_n\circ\pi_2)$, for all $\alpha\in\cO_X$;
			
			\item [d)] $(g+h)_D=g_D+h_D$, for all $g,h\in\cOXp$. 
		\end{enumerate}
	\end{lemma}
	
	In \cite{SG}, we see that it is possible to obtain a set of generators for $\cM_D$ from a set of generators of $\cM$.

	\begin{proposition}[\cite{SG}, Proposition 2.3]\label{P2}
		Suppose that $\cM$ is generated by $\{h_1,\ldots,h_r \}$. Then, the following sets are generators of $\cM_D$:
		\begin{enumerate}
			
			\item [a)] $\cB=\{(h_j)_D\}_{j=1}^r \cup \{(0_{\cO_{X\times X}^{p}},(z_i\circ\pi_1-z_i\circ\pi_2)(h_j\circ\pi_2))\}{_{j=1}^{r}}_{i=1}^n$;
			
			\item [b)] $\cB'=\{(h_j)_D\}_{j=1}^r  \cup \{((z_i\circ\pi_1-z_i\circ\pi_2)(h_j\circ\pi_1),0_{\cO_{X\times X}^{p}})\}{_{j=1}^{r}}_{i=1}^n$;
			
			\item [c)] $\cB''=\{(h_j)_D\}_{j=1}^r  \cup \{(z_ih_j)_D\}{_{j=1}^{r}}_{i=1}^n$.
			
		\end{enumerate}
		
	\end{proposition}

	The following proposition establishes a connection between the integral closure of the double of a module and the integral closure of the original module.
	
	\begin{proposition}[\cite{SG}, Proposition 2.9]\label{P12}
		Let $h\in\cOXp$. If $h_D\in\overline{\cM_D}$ at $(x,x')$ then $h\in\overline{\cM}$ at $x$ and $x'$.
	\end{proposition}

	In the following theorem, we calculate the generic rank of the double of a module.
	
	\begin{theorem}[\cite{SG}, Proposition 2.5]\label{T2.9}
		Let $(X,x)$ be an irreducible analytic complex germ of dimension $d\ge 1$, and $M\subseteq\cO_{X,x}^p$ an $\cO_{X,x}$-submodule of generic rank $k$. Then $M_D$ has generic rank $2k$. 
	\end{theorem}
	
	\begin{corollary}[\cite{SG}, Corollary 2.6]\label{C2.10}
		Let $\{V_i\}$ be the irreducible components of $(X,x)$. For each $i$, if $\cM$ has generic rank $k_i$ on $V_i$ then $\cM_D$ has generic rank $2k_i$ on $V_i\times V_i$. In particular, if $\cM$ has generic rank $k$ on each irreducible component of $X$, then $\cM_D$ has generic rank $2k$ on each irreducible component of $X\times X$.
	\end{corollary} 
	
	We conclude this section by recalling an important result that states that the double of a sheaf of modules $\mathcal{M}$ contains all the information at $(x, x')$ that the stalks of $\mathcal{M}$ have at $x$ and $x'$, provided that $x\neq x'$.
	
	\begin{proposition}[\cite{SG}, Proposition 2.11]\label{P2.13}
		Let $\cM\subseteq\cO_X^p$ be a sheaf of submodules. Consider 
		$(x,x')\in X\times X$ with $x\neq x'$. Then:
		$$\cM_D=(\cM_x\circ\pi_1)\oplus (\cM_{x'}\circ\pi_2)$$ at $(x,x')$.
	\end{proposition}
	
	Proposition \ref{P2.13} offers further justification for the concept of the double: in order to manage the Lipschitz behavior of pairs of tangent planes at two distinct points $x$ and $x'$ in a family $\mathcal{X}$, it is helpful to incorporate the module that determines the tangent hyperplanes at each point into the construction. Moreover, this proposition demonstrates that $JM(\cX)_D$ at $(x,x')$ contains both $JM(\cX)_x$ and $JM(\cX)_{x'}$.

	\section{The Lipschitz saturations of modules}\label{sec3}
	
	We aim to extend the concept of Lipschitz saturation, which was originally defined for ideals in \cite{G2}, to modules. Inspired by the equivalent characterizations of Lipschitz saturation in the case of ideals, we introduce analogous versions of Lipschitz saturation for modules and investigate their interrelations. The following definition is primarily motivated by Theorem 2.3 in \cite{G2}.
	
	Let $X\subseteq \bC^n$ be complex analytic variety, $x\in X$ and let $M$ be an $\cO_{X,x}$-submodule of $\cO_{X,x}^p$.
	
	\begin{definition}
		The {\bf 1-Lipschitz saturation of $M$} is denoted by $M_{S_1}$, and is defined by $$M_{S_1}:=\{h\in\cO_{X,x}^{p}\mid h_D\in\overline{M_D}\}.$$
	\end{definition}
	
	Next, we demonstrate that the first Lipschitz saturation satisfies expected properties related to integral closure.
	
	\begin{proposition}\label{P15} Let $M$ be an $\cO_{X,x}$-submodule of $\cO_{X,x}^p$.
		\begin{enumerate}
			\item [a)] $M_{S_1}$ is an $\cO_{X,x}$-submodule of $\cO_{X,x}^p$;
			\item [b)] $M\subseteq M_{S_1} \subseteq \overline{M}$. In particular, $M$ is a reduction of $M_{S_1}$ and $e(M,M_{S_1})=0$.
		\end{enumerate}
	\end{proposition}
	
	\begin{proof}
		Let $x\in X$ be an arbitrary point.
		
		(a) Let $\alpha\in\cO_{X,x}$ and $h,h'\in M_{S_1}$. Since $h_D\in\overline{\cM_D}$ then by Proposition \ref{P12} (a) we have that $h\in\overline{\cM}$. Thus, $$(0,(\alpha\circ\pi_1-\alpha\circ\pi_2)(h\circ\pi_2))\in\overline{M_D}.$$ Hence
		$(\alpha h+h')_D=(\alpha\circ\pi_1)h_D+(0,(\alpha\circ\pi_1-\alpha\circ\pi_2)(h\circ\pi_2))+h'_D\in\overline{M_D}$.
		
		(b) If $h\in M$ then $h_D\in M_D\subseteq\overline{M_D}$, so $h\in M_{S_1}$. Therefore, $M\subseteq M_{S_1}$. Besides, if $h\in M_{S_1}$ then $h_D\in\overline{M_D}$, and by Proposition \ref{P12} (a) we have that $h\in\overline{M}$. Therefore, $M_{S_1}\subseteq\overline{M}$.
	\end{proof}
	
	To define the second Lipschitz saturation, let us establish some notation. 
	
	For each $\psi:X\rightarrow \mbox{Hom}(\bC^p,\bC)$, $\psi=(\psi_1,\hdots,\psi_p)$, $x\in X$ and $h=(h_1,\hdots,h_p)\in\cO_{X,x}^p$, we define $\psi\cdot h\in\cO_{X,x}$ given by $$(\psi\cdot h)(z):=\sum\limits_{i=1}^{p}\psi_i(z)h_i(z).$$ 
	
	For an $\cO_{X,x}$-submodule $M$ of $\cO_{X,x}^p$, we define $\psi\cdot M$ as the ideal of $\cO_{X,x}$ generated $\{\psi\cdot h \mid h\in \cM\}$. 
	In the next lemma, we prove some fundamental properties of the above operation.
	
	\begin{lemma}\label{L16} With the above notation:
		\begin{enumerate}
			\item [a)] $\psi\cdot(\alpha g+h)=\alpha(\psi\cdot g)+(\psi\cdot h)$, $\forall g,h\in \cO_{X,x}^p$ and $\alpha\in\cO_{X,x}$.
			\item [b)] If $M$ is generated by $\{h_1,\hdots,h_r\}$ then $\psi\cdot M$ is generated by $\{\psi\cdot h_1,\hdots,\psi\cdot h_r\}$.
		\end{enumerate}
	\end{lemma}
	\begin{proof}
		It is easy to see that (b) is a straightforward consequence of (a). Now, write $g=(g_1,\hdots,g_p)$ and $h=(h_1,\hdots,h_p)$. Then, for every $z$ we have $$(\psi\cdot(\alpha g+h))(z)=\sum\limits_{i=1} ^{p}\psi_i(z)(\alpha(z)g_i(z)+h_i(z))$$$$=\alpha(z)\sum\limits_{i=1}^{p}\psi_i(z)g_i(z)+\sum\limits_{i=1}^{p}\psi_i(z)h_i(z)=(\alpha(\psi\cdot g)+(\psi\cdot h))(z).$$
	\end{proof}
	
	Let us fix some notation for the minors of a matrix. Let $k\in\bN$ and let $A$ be a matrix. If $I=(i_1,\hdots,i_k)$ and $J=(j_1,\hdots,j_k)$ are $k$-indexes, $A_{IJ}$ is defined as the $k\times k$ submatrix of $A$ formed by the rows $i_1,\hdots,i_k$ and  columns $j_1,\hdots,j_k$ of $A$. We denote $\cJ_{IJ}(A):=\det(A_{IJ})$.
	
	\begin{lemma}\label{L21}
		Suppose that $M$ has generic rank $k$ in each irreducible component of $X$ at $x$. If $I=(i_1,\hdots,i_k)$ and $J=(j_1,\hdots,j_k)$ are indexes with $j_1=1$ then there exists $\psi: X\rightarrow \mbox{Hom}(\bC^p,\bC)$ such that:
		\begin{enumerate}
			\item [a)] $\psi\cdot h=\cJ_{IJ}(h,M)$, $\forall h\in\cO_{X,x}^p$;
			\item [b)] $\psi\cdot M\subseteq I_k(M)$.
		\end{enumerate}
	\end{lemma}
	
	\begin{proof}
		Let us fix a matrix of generators of $M$, $[M]=$
		$
		\begin{bmatrix}
			|   &           &   | \\
			g_1 & \ldots & g_r \\
			|   &          &    |
		\end{bmatrix}$. We have $\cJ_{IJ}(h,M)=\det [h,M]_{IJ}$, where $$[h,M]_{IJ}=\begin{bmatrix}
			h_{i_1}      &     g_{i_1,j_2-1}  &   \ldots    &     g_{i_1,j_k-1} \\
			\vdots     &         \vdots        &                &         \vdots        \\
			h_{i_k}      &     g_{i_k,j_2-1}  &   \ldots    &     g_{i_k,j_k-1}
		\end{bmatrix}$$ for all $h=(h_1,\hdots,h_p)$. Let $G_{i_l,j_s-1}$ be the $(l,s)$-cofactor of $[h,M]_{IJ}$, for all $l,s\in\{1,\hdots,k\}$. Notice that the $(l,1)$-cofactors $G_{i_l,0}$ do not depend of $h$. Then,  $$\cJ_{IJ}(h,M)=\det[h,M]_{IJ}=\sum\limits_{l=1}^{k}G_{i_l,0}\cdot h_{i_l},$$\noindent for all $h$. Take $\psi: X\rightarrow \mbox{Hom}(\bC^p,\bC)$ given by $(\psi_1,\hdots,\psi_p)$ where $\psi_{i_l}=G_{i_l,0}$, for all $l\in\{1,\hdots,k\}$, and $\psi_j=0$, for every index $j$ off $I$. Thus, for all $h\in\cO_{X,x}^p$ we get $\cJ_{IJ}(h,M)=\psi_I\cdot h_I=\psi\cdot h$.
		
		Now, let $g\in M$ be arbitrary. Then $(g,M)=M$. By (a) we have $\psi\cdot g=\cJ_{I,J}(g,M)\in I_k((g,M))=I_k(M)$, hence $\psi\cdot M\subseteq I_k(M)$.
	\end{proof}
	
	{\bf Remark:} It is enough work with indexes $I=(i_1,\hdots,i_k)$ and $J=(j_1,\hdots,j_k)$ such that $i_1<\hdots<i_k$ and $j_1<\hdots<j_k$. 
	
	Proposition \ref{proposition G} is a characterization for the integral closure of modules using the integral closure of ideals. In the next result, we obtain a new characterization of this type.
	
	\begin{proposition}\label{P18}
		Let $x\in X$, $h\in\cO_{X,x}^p$ and suppose $M$ has generic rank $k$ on each irreducible component of $X$. Then, \begin{center}$h\in\overline{M}$ if, and only if, $\psi\cdot h\in\overline{\psi\cdot M}$,\end{center}\noindent $\forall \psi:X\rightarrow \mbox{Hom}(\bC^p,\bC)$.
	\end{proposition}
	
	\begin{proof}

		$(\implies)$ It is a straightforward consequence of the curve criterion.
		
		$(\impliedby)$ The proof now use 1.7 and 1.8 of \cite{G3}.
		
		By these results it is enough to check that $\cJ_{IJ}(h,M)\in\overline{I_k(M)}$, for all indexes $I$ and $J$. Write $I=(i_1,\hdots,i_k)$ and $J=(j_1,\hdots,j_k)$. Let $$[M]=\begin{bmatrix}
			|   &           &   | \\
			g_1 & \ldots & g_r \\
			|   &          &    |
		\end{bmatrix}$$\noindent be a matrix of generators of $M$. Write $h=(h_1,\hdots,h_p)$. Then, $[h,M]=$
		$
		\begin{bmatrix}
			|    &  |   &           &   | \\
			h   & g_1 & \ldots & g_r \\
			|   &  |   &          &    |
		\end{bmatrix}$. If $j_1>1$ then $\cJ_{IJ}(h,M)$ is a $k\times k$ minor taken only among the generators of $M$, hence $$\cJ_{IJ}(h,M)\in I_k(M)\subseteq\overline{I_k(M)}.$$ 
		
		Now suppose that $j_1=1$. Thus, Lemma \ref{L21} ensures there exists $\psi:X\rightarrow \mbox{Hom}(\bC^p,\bC)$ such that $\psi\cdot h=\cJ_{I,J}(h,M)$ and $\psi\cdot M\subseteq I_k(M)$. By hypothesis we have $\psi\cdot h\in\overline{\psi\cdot M}$, therefore $$\cJ_{I,J}(h,M)=\psi\cdot h\in\overline{\psi\cdot M}\subseteq\overline{I_k(M)}.$$
	\end{proof}
	
	Let $x\in X$ and let $M$ be an $\cO_{X,x}$-submodule of $\cO_{X,x}^p$.
	
	\begin{definition}
		The {\bf 2-Lipschitz saturation of $M$ } is denoted by $M_{S_2}$, and is defined by $$M_{S_2}:=\{h\in\cO_{X,x}^{p}\mid\psi\cdot h\in(\psi\cdot \cM)_S, \forall \psi:X\rightarrow \textrm{Hom}(\bC^p,\bC)\}.$$
	\end{definition}
	
	Now, we prove that the second Lipschitz saturation satisfies the expected properties related to integral closure.
	
	\begin{proposition}\label{P19} Let $M$ be an $\cO_{X,x}$-submodule of $\cO_{X,x}^p$, $x\in X$.
		\begin{enumerate}
			\item [a)] $M_{S_2}$ is an $\cO_{X,x}$-submodule of $\cO_{X,x}^p$;
			
			\item [b)] $M\subseteq M_{S_2} \subseteq \overline{M}$. In particular, $M$ is a reduction of $M_{S_2}$ and  $e(M,M_{S_2})=0$.
		\end{enumerate}
	\end{proposition}
	
	\begin{proof}
		
		(a) Let $h,h'\in M_{S_2}$ and $\alpha\in\cO_X$, and let $\psi:X\rightarrow \mbox{Hom}(\bC^p,\bC)\}$ arbitrary. Then, $\psi\cdot h,\psi\cdot h'\in(\psi\cdot M)_S$, and by Lemma \ref{L16} (a) we have \begin{center}$\psi\cdot(\alpha h+h')=\alpha(\psi\cdot h)+\psi\cdot h'\in(\psi\cdot M)_S$.\end{center} Therefore, $\alpha h+h'\in M_{S_2}$.
		
		(b) If $h\in M$ then $\psi\cdot h\in\psi\cdot M\subseteq(\psi\cdot M)_S$, $\forall\psi:X\rightarrow \mbox{Hom}(\bC^p,\bC)$. Hence, $h\in M_{S_2}$ and $M\subseteq M_{S_2}$. Now, let $h\in M_{S_2}$. Then, $$\psi\cdot h\in(\psi\cdot M)_S\subseteq\overline{\psi\cdot M},$$\noindent $\forall\psi:X\rightarrow \mbox{Hom}(\bC^p,\bC)$. Thus Proposition \ref{P18} ensures $h\in\overline{M}$.
	\end{proof}
	
	Next, we begin to compare the first and second  Lipschitz saturations.
	
	\begin{proposition}\label{P23}
		Let $x\in X$ and let $M$ be an $\cO_{X,x}$-submodule of $\cO_{X,x}^p$. Then $$M_{S_1}\subseteq M_{S_2}.$$
	\end{proposition}
	
	\begin{proof}
		Let $h=(h_1,\hdots,h_p)\in M_{S_1}$. Then $h_D\in \overline{M_D}$. We need to check that $(\psi\cdot h)_D\in\overline{(\psi\cdot M)_D}$, for all $\psi:X\rightarrow \mbox{Hom}(\bC^p,\bC)$. Let $\phi=(\phi_1,\phi_2):(\bC,0)\rightarrow(X\times X,(x,x))$ be an arbitrary analytic curve. Since $h_D\in\overline{M_D}$ then we can write $$h_D\circ\phi=\sum\limits_{j}\alpha_j\phi^*((g_j)_D)$$ with $g_j=(g_{1j},\hdots,g_{pj})\in M$ and $\alpha_j\in\cO_{\bC,0}$ for all $j$. Looking to the above equation and comparing the $2p$ coordinates we conclude that \begin{center}$h_i\circ\phi_1=\sum\limits_{j}\alpha_j(g_{ij}\circ\phi_1)$ and $h_i\circ\phi_2=\sum\limits_{j}\alpha_j(g_{ij}\circ\phi_2)$,\end{center}\noindent for all $i\in\{1,\hdots,p\}$. Let $\psi_1,\hdots,\psi_p$ be the coordinate functions of $\psi$.
		
		Thus: $$(\psi\cdot h)_D\circ\phi=(\sum\limits_{i}(\psi_i\circ\phi_1)\cdot(h_i\circ\phi_1),\sum\limits_{i}(\psi_i\circ\phi_2)\cdot(h_i\circ\phi_2))$$$$=(\sum\limits_{i,j}(\psi_i\circ\phi_1)\alpha_j(g_{ij}\circ\phi_1),\sum\limits_{i,j}(\psi_i\circ\phi_2)\alpha_j(g_{ij}\circ\phi_2))$$$$=\sum\limits_{j}\alpha_j((\psi\cdot g_j)_D\circ\phi)\in(\psi\cdot M)_D\circ\phi.$$
	\end{proof}

	The next definition of Lipschitz saturation is motivated by Proposition \ref{proposition G}, which relates $\overline{M}$ and $\overline{I_k(M)}$.
	
	Let $x\in X$ and let $M$ be an $\cO_{X,x}$-submodule of $\cO_{X,x}^p$.
	
	\begin{definition}
		Suppose that $M$ has generic rank $k$ on each irreducible component of $X$. The \textbf{3-Lipschitz saturation of $M$} is denoted by $M_{S_3}$, and is defined by $$M_{S_3}:=\{h\in\cO_{X,x}^{p}\mid I_k(h,M)\subseteq(I_k(M))_S\}.$$
	\end{definition}
	
	Let us prove that the third Lipschitz saturation satisfies the expected properties related to integral closure.
	
	\begin{proposition} Let $x\in X$ and suppose that $M\subseteq\cO_{X,x}^p$ is an $\cO_{X,x}$-submodule of generic rank $k$ on each irreducible component of $X$. Then:
		\begin{enumerate}
			\item [a)] $M_{S_3}$ is an $\cO_{X,x}$-submodule of $\cO_{X,x}^p$;
			
			\item [b)] $M\subseteq M_{S_3} \subseteq \overline{M}$. In particular, $M$ is a reduction of $M_{S_3}$ and $e(M,M_{S_3})=0$.
		\end{enumerate}
	\end{proposition}
	
	\begin{proof}
		(a) Let $g,h\in M_{S_3}$ and $\alpha\in\cO_{X,x}$. By the basic properties of determinants we have that $$I_k(\alpha g+h,M)\subseteq \alpha I_k(g,M)+I_k(h,M)\subseteq(I_k(M))_S.$$ 
		
		Hence, $\alpha g+h\in M_{S_3}$.	 
		
		(b) Since $I_k(M)\subseteq (I_k(M))_S$ then $M\subseteq M_{S_3}$. Now, let $h\in M_{S_3}$. Thus, $I_k(h,M)\subseteq (I_k(M))_S\subseteq\overline{I_k(M)}$ which implies $h\in\overline{M}$.
	\end{proof}
	
	In the next result, we show that the second Lipschitz saturation is stronger than the third one.
	
	\begin{proposition}\label{P22}
		Let $x\in X$ and suppose that $M\subseteq\cO_{X,x}^p$ is an $\cO_{X,x}$-submodule of generic rank $k$ on each irreducible component of $X$. Then $$M_{S_2}\subseteq M_{S_3}.$$
	\end{proposition}
	
	\begin{proof}
		Suppose $h\in M_{S_2}$. By Theorem 2.3 of \cite{G2} it is enough to prove that $(I_k(h,M))_D\subseteq\overline{(I_k(M))_D}$. So, we have to prove that all the generators $(\cJ_{IJ}(h,M))_D$ are in $\overline{(I_k(M))_D}$, for all $i\in\{1,\hdots,n\}$ and for all indexes $I$ and $J$. Write $I=(i_1,\hdots,i_k)$ and $J=(j_1,\hdots,j_k)$. We have two cases. 
		
		Suppose $j_1>1$. Then, as we saw before, $\cJ_{IJ}(h,M)\in I_k(M)$, so $(\cJ_{IJ}(h,M))_D\in (I_k(M))_D$, in particular, $(\cJ_{IJ}(h,M))_D\in \overline{(I_k(M))_D}$.
		
		Suppose $j_1=1$. By Lemma \ref{L21} there exists $\psi:X\rightarrow \mbox{Hom}(\bC^p,\bC)$ such that $\psi\cdot h=\cJ_{IJ}(h,M)$ and $\psi\cdot M\subseteq I_k(M)$. Since $h$ is in the 2-Lipschitz saturation of $M$ then $\psi\cdot h\in(\psi\cdot M)_S$, which is equivalent to $(\psi\cdot h)_D\in\overline{(\psi\cdot M)_D}$. Thus, $$(\cJ_{I,J}(h,M))_D=(\psi\cdot h)_D\in\overline{(\psi\cdot M)_D}\subseteq\overline{(I_k(M))_D}.$$
	\end{proof}
	
	The next result shows the third saturation is closely related to the pullback of the saturated blow-up map.
	
	\begin{proposition}\label{P20}
		Let $x\in X$ and suppose that $M\subseteq\cO_{X,x}^p$ is an $\cO_{X,x}$-submodule of generic rank $k$ on each irreducible component of $X$. Let $\pi: SB_{I_{k}(M)}(X)\rightarrow B_{I_{k}(M)}(X)$, $p: B_{I_{k}(M)}(X)\rightarrow X$ and $\pi_S=p\circ\pi$ be the projections maps. Let $h\in\cO_{X,x}^p$. Then: \begin{center} $h\in M_{S_3}$ if and only if $\pi_S^{*}(h)\in\pi_S^{*}(M)$.\end{center}
	\end{proposition}
	
	\begin{proof}
		Fix a set of generators $\{g_1,\hdots,g_r\}$ of $M$.
		We work at $x'\in E$, $E=\pi^{-1}(E_B)$, $E_B$ the exceptional divisor of $B_{I_{k}(M)}(X)$. 
		
		Suppose first that $I_k(h,M)\subseteq (I_k(M))_S$. Let $\pi_S^{*}(\cJ_{I,J}(M))$ be a local generator of the principal ideal $\pi_S^{*}(I_k(M))$. Then, by Cramer's rule we can write $$(\cJ_{I,J}(M)\circ\pi_S)(h_I\circ\pi_S)=\sum\limits_{j\in J}(\cJ_{I,j}(h_I,M_I)\circ\pi_S)(m_{I,j}\circ\pi_S) \eqno(1)$$ with $m_j\in M$. We claim in fact that $$(\cJ_{I,J}(M)\circ\pi_S)(h\circ\pi_S)=\sum\limits_{j\in J}(\cJ_{I,j}(h_I,M_I)\circ\pi_S)(m_{j}\circ\pi_S).$$ To see this pick a curve $\phi:(\bC,0)\rightarrow(SB_{I_k(M)}(X),x')$ and choose $\phi$ so that the rank of $\pi_S^*(M)|_{\phi}$ is generically $k$. Since by hypothesis $h\in\overline{M}$ then $h\circ\pi_S\circ\phi\in\phi^*(\pi_S^*(M))$. So, the element $$\star=h\circ\pi_S\circ\phi-\sum\limits_{j\in J}(\frac{\cJ_{I,j}(h_I,M_I)\circ\pi_S}{\cJ_{I,J}(M)\circ\pi_S}\circ\phi)(m_{j}\circ\pi_S\circ\phi)\in\phi^*(\pi_S^*(M))$$ because the above quotients by hypothesis are regular functions on $SB_{I_k(M)}(X)$. By equation (1), the above element has $0$ for the entries indexed by $I$. So, if $\star$ is not zero then $\phi^*(\pi_S^*(M))$ has rank at least $k+1$, which is a contradiction. Since the images of $\phi$ fill up a Z-open set, it follows that $\star$ is locally zero. Therefore, $$h\circ\pi_S=\sum\limits_{j\in J}(\frac{\cJ_{I,j}(h_I,M_I)\circ\pi_S}{\cJ_{I,J}(M)\circ\pi_S})(m_{j}\circ\pi_S)\in\pi_S^*(M).$$
		
		Conversely, suppose that $\pi_S^{*}(h)\in\pi_S^{*}(M)$. Then we can write $$h\circ\pi_S=\sum\limits_{j=1}^{r}\alpha_j(g_j\circ\pi_S) \eqno(2)$$
		where $\alpha_j\in\widetilde{\cO}_{B_{I_k(M)}(X),x'}$, $\forall j\in\{1,\hdots,r\}$. 
		
		\noindent Let $c$ be an arbitrary generator of $I_k(h,M)$. Then we can write $c=\det[h,M]_{IJ}$ for some $k$-indexes $I$ and $J$. If the $k$-index $J$ does not pick the first column of $[h,M]$, then $c$ is a $k\times k$ minor of the matrix $[M]$, hence $c\in I_k(M)\subseteq (I_k(M))_S$ and we are done. Thus we may assume that $j_1=1$. Then, we can write $$c=\det\left[\begin{matrix}
			h_I  &  (g_{j_1})_I  &\hdots &(g_{j_{k-1}})_I
		\end{matrix}\right].$$
		Applying the pullback of the projection map we have $$\pi_S^*(c)=\det\left[\begin{matrix}
			h_I\circ\pi_S  &  (g_{j_1})_I\circ\pi_S  &\hdots &(g_{j_{k-1}})_I\circ\pi_S
		\end{matrix}\right].$$
		
		\noindent By the equation (2), if we look only at the entries of the $k$-index $I$, we get $$h_I\circ\pi_S=\sum\limits_{j=1}^{r}\alpha_j((g_j)_I\circ\pi_S).$$
		
		\noindent Using the linearity of the determinant in the first column, we conclude that $\pi_S^*(c)=\sum\limits_{j=1}^{r}\alpha_j(v_j\circ\pi_S)$, where $v_j:=\det\left[\begin{matrix}
			(g_j)_I  &  (g_{j_1})_I  &\hdots &(g_{j_{k-1}})_I
		\end{matrix}\right],$ for all $j\in\{1,\hdots,r\}$. Clearly, $v_j\in I_k(M)$, $\forall j\in\{1,\hdots,r\}$. Hence, $\pi_S^*(c)\in\pi_S^*(I_k(M))$ and by the definition of the Lipschitz saturation of an ideal, we conclude that $c\in(I_k(M))_S$. 
	\end{proof}

	
	\section{Generic equivalence}
	
	If $\cM$ is a sheaf of $\cO_X$-submodules of $\cO_{X}^p$, we can construct the sheaf $\cM_{S_i}$ of $\cO_X$-submodules of $\cO_{X}^p$ associated to the $i$-th Lipschitz saturation, $i\in\{1,2,3\}$, in a way that $$(\cM_{S_i})_x=(\cM_x)_{S_i}\subseteq\cO_{X,x}^p,$$\noindent  $\forall x\in X$. Proposition \ref{P23} and \ref{P22} tell us that $\cM_{S_1}\subseteq\cM_{S_2}\subseteq\cM_{S_3}$. From now on, we look for conditions so that the above notions of Lipschitz saturations are equivalent.
	
	Naturally, the first step in this direction is to seek conditions for an element in $M_{S_3}$ to be in $M_{S_1}$. The next proposition gives us such a condition.
	
	\begin{proposition}\label{T4.17}
		Let $x\in X$ and suppose that $M\subseteq\cO_{X,x}^p$ is an $\cO_{X,x}$-submodule of generic rank $k$ on each irreducible component of $X$, and let $h\in\cO_{X,x}^p$. Assume that there exists an ideal $I$ of $\cO_{X\times X,(x,x)}^{2p}$ such that
		
		\begin{itemize}
			\item $ II_2((I_k(\cM))_D)\subseteq\overline{I_{2k}(M_D)}$;
			
			\item $I_{2k}(h_D,M_D)\subseteq\overline{II_2((I_k(h,M))_D)}$.
		\end{itemize}
		
		If $h\in M_{S_3}$ then $h\in M_{S_1}$.
	\end{proposition}
	
	\begin{proof}
		Suppose $h\in M_{S_3}$. Then, $(I_k(h,M))_D\subseteq\overline{(I_k(M))_D}$ and so $\overline{(I_k(h,M))_D}=\overline{(I_k(M))_D}$.
		
		Let us prove that $I_{2k}(h_D,M_D)\subseteq\overline{I_{2k}(M_D)}$. In fact, let $\phi:(\bC,0)\rightarrow(X\times X,(x,x))$ be an arbitrary analytic curve. Since $$\overline{(I_k(h,M))_D}=\overline{(I_k(M))_D}$$\noindent then $\phi^*((I_k(h,M))_D)=\phi^*((I_k(M))_D)$. Using the inclusions of the hypothesis and the curve criterion for the integral closure of modules, we have $$\phi^*(I_{2k}(h_D, M_D))\subseteq\phi^*(I\cdot I_2((I_k(h, M))_D))=\phi^*(I)\cdot I_2(\phi^*((I_k(M))_D))$$$$=\phi^*(I\cdot I_2((I_k(M))_D))\subseteq\phi^*(I_{2k}(M_D)).$$
		
		By Theorem \ref{T2.9}, we have that $M_D$ has generic rank $2k$ on each component of $X\times X$. Hence, by Proposition \ref{proposition G} we conclude that $h_D\in\overline{M_D}$, i.e, $h\in M_{S_1}$.  	 
	\end{proof}	
	
	In what follows, we prove some auxiliary results that will be useful for applying the previous proposition to sheaves of submodules of $\cO_{X}^p$. First, it is convenient for us to establish a notation beforehand. For each object $A$, we denote $A=A\circ\pi_1$ and $A'=A'\circ\pi_2$. For $k$-indexes $I,J$, we denote by $M_{IJ}$ the submatrix of $[M]$ defined by these $k$-indexes.
	
	\begin{lemma}
		Let $x\in X$, $M$ an $\cO_{X,x}$-submodule of $\cO_{X,x}^p$ and $k\in\bN$. Then $$(z_{t_1}-z_{t_1}')\cdots(z_{t_k}-z_{t_k}')\det(M_{IJ})\det(M_{KL}')\in I_{2k}(M_D) \mbox{ at }(x,x)$$ for every $t_1,\hdots,t_k\in\{1,\hdots,n\}$ and $k$-indexes $I,J,K,L$.
		
	\end{lemma}
	
	\begin{proof}
		Write $M_{KL}=(g_{rs})$. The $2k\times 2k$ matrix $$\left[\begin{matrix}
			
			M_{IJ}   &  |  & 0_{k\times k} \\
			------- &      &  ------------\\
			
			M_{IJ}'   &   |   &  
			\begin{matrix}
				(z_{t_1}-z_{t_1}')g'_{11}    &   \hdots   & (z_{t_k}-z_{t_k}')g'_{1k}\\
				\vdots                &         &      \vdots    \\
				(z_{t_1}-z_{t_1}')g'_{k1}    &   \hdots   & (z_{t_k}-z_{t_k}')g'_{kk}       
			\end{matrix}

		\end{matrix}\right]$$ is a submatrix of $[M_D]$. 
		
		So, its determinant, which is $(z_{t_1}-z_{t_1}')\cdots(z_{t_k}-z_{t_k}')\det(M_{IJ})\det(M_{KL}')$, belongs to $I_{2k}(M_D)$.	
	\end{proof}
	
	The next result gives us the first condition of Proposition \ref{T4.17} in terms of the ideal coming from the diagonal.
	
	\begin{lemma}\label{L4.20}
		Let $x\in X$, $M$ an $\cO_{X,x}$-submodule of $\cO_{X,x}^p$ and $k\in\bN$. Then: 
		\begin{itemize}
			\item[a)] $I_{\Delta}^{k}I_2((I_k(M))_D)\subseteq I_{2k}(M_D)$ at $(x,x)$;
			\item[b)] If $I_k(M)$ is principal then $I_{\Delta}^{k-1}I_2((I_k(M))_D)\subseteq I_{2k}(M_D)$ at $(x,x)$.
		\end{itemize}
	\end{lemma}
	
	\begin{proof}
		(a) We have that $$[(I_k(M))_D] =\left[\begin{matrix}
			
			\det(M_{IJ})    &  \hdots   &     0 \\
			\det(M_{IJ}')  & \hdots &   (z_i-z_i')\det(M_{KL}')\\ 	              
		\end{matrix}\right]$$
		varying the $k$-indexes $I,J,K,L$ and $i\in\{1,\hdots,n\}$. Thus, the desired inclusion is a straightforward consequence of the previous lemma.
		
		(b) Since $I_k(M)$ is principal then there exist $k$-indexes $I,J$ such that $g=\det(M_{IJ})$ and $I_k(M)$ is generated by $\{g\}$.Thus we can write $$[(I_k(M))_D] =\left[\begin{matrix}
			
			g    &      0 \\
			g'   &   (z_i-z_i')g'\\ 	              
		\end{matrix}\right]$$
		varying $i\in\{1,\hdots,n\}$. So, in this case $I_2((I_k(M))_D)$ is generated by $\{gg'(z_i-z_i')\mid i\in\{1,\hdots,n\}\}$. By previous lemma we have $$[(z_{t_1}-z_{t_1}')\cdots(z_{t_{k-1}}-z_{t_{k-1}}')].[g.g'(z_i-z_i')]\in I_{2k}(M_D),$$
		
		\noindent for all $t_1,\hdots,t_{k-1},i\in\{1,\hdots,n\}$. Therefore, $$I_{\Delta}^{k-1}I_2((I_k(M))_D)\subseteq I_{2k}(M_D).$$ 
	\end{proof}
	
	Now, we start to get conditions in order to obtain the second condition of Proposition \ref{T4.17}.
	
	\begin{lemma}
		Let $x\in X$, $M$ an $\cO_{X,x}$-submodule of $\cO_{X,x}^p$ and $k\in\bN$. Denote by $[\tilde{M}]$ the matrix such that $$[M_D]=\left[\begin{matrix}
			[M]        &     0 \\
			[M]'        &    [\tilde{M}]
		\end{matrix}\right].$$
		
		Let $\cI_{2k}(M_D)$ be the subideal of $I_{2k}(M_D)$ generated by $$\{\det(M_{IJ})\det(\tilde{M}_{KL})\mid I,J,K,L\mbox { are }k\mbox{-indexes}\}.$$
		
		Then, $$\cI_{2k}(M_D)\subseteq I_{\Delta}^{k-1}I_2((I_k(M))_D).$$
	\end{lemma}
	
	\begin{proof}
		It suffices to prove that each generator of $\cI_{2k}(M_D)$ belongs to $I_{\Delta}^{k-1}I_2((I_k(M))_D)$.
		
		The columns of $\tilde{M}_{KL}$ are columns of $[M]'$ possibly in a different order with terms of type $z_i-z_i'$ multiplied on each column. If the $k$-indexes $K,L$ pick repeated columns then $\det(\tilde{M}_{KL})=0$ and we are done. If this does not occur, then 
		$$\det(\tilde{M}_{KL})=(z_{i_1}-z_{i_1}')\cdots(z_{i_k}-z_{i_k}')(\pm \det(M'_{K'L'}))$$ for some reorganization $k$-indexes $K',L'$, where $i_1,\hdots,i_k\in\{1,\hdots,n\}$ are the indexes which $z_{i_1}-z_{i_1}',\hdots,z_{i_k}-z_{i_k}'$ appear on each of the $k$ columns of $\tilde{M}_{KL}$. Since $\pm(z_{i_1}-z_{i_1}')\cdots(z_{i_{k-1}}-z_{i_{k-1}}')\in I_{\Delta}^{k-1}$ and $$\det(M_{IJ})(z_{i_k}-z_{i_k}')\det(M'_{KL})\in I_2((I_k(M))_D)$$\noindent then $\det(M_{IJ})\det(\tilde{M}_{KL})\in I_{\Delta}^{k-1}I_2((I_k(M))_D)$.
	\end{proof}
	
	The next result ensures that any free submodule satisfies the second condition of Proposition \ref{T4.17} for a suitable power of the ideal coming from the diagonal.
	
	\begin{lemma}
		Let $x\in X$ and $M$ a free $\cO_{X,x}$-submodule of $\cO_{X,x}^p$ of rank $k\in\bN$. Then $$I_{2k}(M_D)\subseteq I_{\Delta}^{k-1}I_2((I_k(M))_D).$$
	\end{lemma}
	
	\begin{proof}
		We have that $$[M_D]=\left[
		\begin{matrix}
			[M]_{p\times k}   &    |    &     0   \\
			------        &    |    &  -------  \\
			[M]'_{p\times k} & |    &  (z_i-z_i')[M]'
			
		\end{matrix}
		\right] $$
		varying $i\in\{1,\hdots,n\}$.
		
		Let $d\in I_{2k}(M_D)$ be an arbitrary generator. Then $d=\det N$ where $N$ is a $2k\times 2k$ submatrix of $[M_D]$.
		
		(i) Suppose first that there are $k+t$ columns of $N$ taken on the part $$\left[\begin{matrix}
			0 \\
			(z_i-z_i')[M]'
		\end{matrix}\right]$$
		of $[M_D]$, with $1\leq t \leq k$. Then we can write $$N=\left[\begin{matrix}
			(M_{IJ})_{k\times(k-t)}    &    |    &   0_{k\times({k+t})}\\
			--------         &    |    &   --------\\
			(M_{IJ})'_{k\times(k-t)}   &    |    &  (\tilde{M})_{KL}
		\end{matrix}\right]$$
		
		$$=\left[\begin{matrix}
			M_{IJ}   &    0_{k\times t}  &   |  &   0_{k\times k}\\
			---    & ---        &   |  &  ------\\
			M_{IJ}'  &     *             &   |  &    **
		\end{matrix}\right]$$ for some matrices * and **, where ** is a square matrix of size $k\times k$. So in this case we have that $$d=\det N=\det\left[\begin{matrix}
			M_{IJ}  &   0_{k\times t}
		\end{matrix}\right]\cdot\det(**)=0\in I_{\Delta}^{k-1}I_2((I_k(M))_D).$$  	
		
		(ii) Now suppose that we have exactly $k$ columns of $N$ taken on the part $$\left[\begin{matrix}
			0 \\
			(z_i-z_i')[M]'
		\end{matrix}\right]$$
		of $[M_D]$. Then we can write $$N=\left[\begin{matrix}
			M_{IJ}    &    |      &     0_{k\times k}  \\
			------    &    |      &     ------ \\
			M_{IJ}'   &    |      &     \tilde{M}_{KL}
		\end{matrix}\right].$$
		
		Thus, $d=\det(M_{IJ})\det(\tilde{M}_{KL})\in\cI_{2k}(M_D)$ and by previous lemma we conclude that $d\in I_{\Delta}^{k-1}I_2((I_k(M))_D)$.
	\end{proof}
	
	As a consequence of the previous lemma, we have a result that states the second condition of Proposition \ref{T4.17} holds locally in a dense Zariski open subset of $X$ in a sheaf of submodules of $\cOXp$.
	
	\begin{lemma}\label{L4.23}
		Let $\cM$ be an $\cO_X$-submodule of $\cOXp$ of generic rank $k$ on each component of $X$. Then there exists a dense Zariski open subset $U$ of $X$ such that:
		
		\begin{itemize}
			\item $U\cap V$ is a dense Zariski open subset of $V$, $\forall V$ component of $X$;
			
			\item $I_k(\cM)=\cO_{X,x}$ at $x, \forall x\in U$;
			
			\item $I_{2k}(\cM_D)\subseteq I_{\Delta}^{k-1}I_2((I_k(\cM))_D)$ at $(x,x)$, for all $x\in U$. 
		\end{itemize}

	\end{lemma}
	
	\begin{proof}
		Consider $[\cM]$ a matrix of generators of $\cM$. Using Cramer's rule, we can choose $k$ $\cO_X$-linear independents columns of $[\cM]$ such that these columns generates $\cM$ in such a dense Zariski open subset $U$ of $X$ and $I_k(\cM)=\cO_X$ along $U$. Let $\cM_k$ be the $\cO_X$-submodule of $\cOXp$ generated by the columns chosen above. Thus, given $x\in U$, we have that $\cM_x=(\cM_k)_x$ is a free $\cO_{X,x}$-submodule of $\cO_{X,x}^p$ of rank $k$ and the desired inclusion is a consequence of the previous lemma.  
	\end{proof}
	
	Before we state the main theorem of this section, we prove that the generic rank of $(h,\cM)$ is the same as the generic rank of $\cM$, provided $h\in\overline{\cM}$.
	
	\begin{lemma}
		Let $\cM$ be an $\cO_{X}$-submodule of $\cO_{X}^p$ of generic rank $k$ on each component of $X$ at $x\in X$. If $h\in\overline{\cM}$ at $x$ then $(h,\cM)$ also has generic rank $k$ on each component of $X$ at $x$.
	\end{lemma}
	
	\begin{proof}
		Since $h\in\overline{\cM_x}$ then $h\in\overline{\cM_{x'}}$, for all $x'$ in an open neighborhood $U$ of $x$ in $X$. Suppose the rank of $(h,\cM)$ is $k+1$. Then $h\circ x'\notin \cM\circ x'$. But, $x'$ is a constant curve, so $h\notin\overline{\cM_{x'}}$, contradiction.
	\end{proof}
	
	The next theorem is the main objective of this work and establishes that the three defined Lipschitz saturations are generically equivalent on a sheaf of submodules of $\cOXp$.
	
	\begin{theorem}
		Let $\cM$ be an $\cO_X$-submodule of $\cO_X^p$ of generic rank $k$ on each component of $X$. Then there exists a dense Zariski open subset $U$ of $X$ 
		such that	$$ \cM_{S_1}=\cM_{S_2}=\cM_{S_3}$$ along $U$. 
	\end{theorem}
	
	\begin{proof}
		
		Since the equalities mentioned above can be verified on each irreducible component of $X$, we may assume that $X$ is irreducible. We are dealing with a complex analytic variety $X$, so it is well-known that $\cO_X$ is a noetherian sheaf of rings. Furthermore, Oka's Coherence Theorem guarantees that $\cO_X^p$ is a coherent $\cO_X$-module. Consequently, there exists a non-empty open subset $W$ of $X$ (which is dense, as $X$ is irreducible) such that $\cM_{S_3}\mid_{W}$ is of finite type over $\cO_X$. Choose $h_1,...,h_r$ generators of $\cM_{S_3}$ on W.

		Thus, $h_i\in\overline{\cM}$ at points of $W$, $\forall i\in\{1,\hdots,r\}$ then by previous lemma $(h_i,\cM)$ also has generic rank $k$ on $X$ at points of $W$, $\forall i\in\{1,\hdots,r\}$. By Lemma \ref{L4.23}, for each $i\in\{1,\hdots,r\}$ there exists a dense Zariski open subset $U_i$ of $X$ such that $$I_{2k}((h_i)_D,\cM_D)\subseteq I_{\Delta}^{k-1}I_2((I_k(h_i,\cM))_D)\mbox{ at }(x,x),$$ $\forall x\in U_i\cap W$. Also by Lemma \ref{L4.23} there exists a dense Zariski open subset $U_0$ of $X$ such that $I_k(\cM)=\cO_{X,x}$ which is principal at $x$, $\forall x\in U_0$. By Lemma \ref{L4.20}(b), we conclude that $$I_{\Delta}^{k-1}I_2((I_k(\cM))_D)\subseteq I_{2k}(\cM_D)\mbox{ at }(x,x),$$
		$\forall x\in U_0$. Taking $$U:=\left(\bigcap\limits_{j=0}^r U_j\right)\cap W,$$ \noindent it follows that $U$ is a dense Zariski open subset of $X$. We already know that $\cM_{S_1}\subseteq \cM_{S_2}\subseteq \cM_{S_3}$ at every point of $X$, in particular, along $U$. Let us prove that $\cM_{S_3}\subseteq \cM_{S_1}$ along $U$. In fact, let $x\in U$. Given an arbitrary $i\in\{1,\hdots,r\}$, since $x\in U_0$ and $x\in U_i$ then: 
		\begin{itemize}
			\item $I_{\Delta}^{k-1}I_2((I_k(\cM))_D)\subseteq I_{2k}(\cM_D)\mbox{ at }(x,x)$;
			
			\item 	$I_{2k}((h_i)_D,\cM_D)\subseteq I_{\Delta}^{k-1}I_2((I_k(h_i,\cM))_D)\mbox{ at }(x,x)$. 
		\end{itemize} 
		
		Since $h_i\in \cM_{S_3}$ at $x$ then by Proposition \ref{T4.17} we have that $h_i \in \cM_{S_1}$ at $x$, $\forall i\in\{1,\hdots,r\}$. Since $\cM_{S_3}$ at $x$ is generated by $h_1,\hdots,h_r$ at $x$, then $\cM_{S_3}\subseteq \cM_{S_1}$ at $x$.
	\end{proof}
	
	In the next example, we show that, in general, $\cM_{S_3}$ is not a submodule of $\cM_{S_1}$.
	
	\begin{example}\normalfont
		Consider the submodule $M$ of $\cO_2^2$ given by $[M]=\left[\begin{matrix}
			x  &  0  &   y\\
			y  &  x  &   0
		\end{matrix}\right]$ and let $h=(x,3y)$. Thus, $$I_2(M)=\langle  x^2,xy,y^2\rangle=I_2(h,M).$$ 
		
		In particular, $h\in M_{S_3}$. On the other hand, consider the curve $\Phi:(\bC,0)\rightarrow(\bC^2\times\bC^2,(0,0))$ given by $\Phi(t)=(t,\alpha t, t, \beta t)$.
		
		We have that $$[\Phi^*(M_D)]=\left[\begin{matrix}
			t & 0 & \alpha t & 0 & 0 & 0 \\
			\alpha t & t & 0 &  0 & 0 & 0\\
			t & 0 & \beta t & (\alpha-\beta)t^2 & 0 & \beta(\alpha-\beta)t^2\\
			\beta t & t & 0 & \beta(\alpha-\beta)t^2 & (\alpha-\beta)t^2 & 0
		\end{matrix}\right]$$
		
		\noindent and $[\Phi^*(h_D)]=\left[\begin{matrix}
			t\\
			3\alpha t \\
			t\\
			3\beta t
		\end{matrix}\right]$. Let $D$ be the first $4\times 3$ submatrix of $[\Phi^*(M_D)]$. The complementary submatrix has degree 2, so we can ignore its terms. Notice that:
		
		$[\Phi^*(h_D)]-D\cdot e_1=\left[\begin{matrix}
			0\\
			2\alpha t \\
			0\\
			2\beta t
		\end{matrix}\right]$ and $[\Phi^*(h_D)]-D\cdot (2\alpha e_2)=\left[\begin{matrix}
			0\\
			0 \\
			0\\
			(2\beta-2\alpha) t
		\end{matrix}\right]$. However, the first $3$ vectors of $D$ generically have rank $3$. Therefore, $h\notin M_{S_1}$.
	\end{example}
	
	Following the author's idea in \cite{G1, SG} of defining Lipschitz infinitesimal conditions for the case of a family of complex analytic varieties, we observe that each previously defined Lipschitz saturation induces a type of infinitesimal condition. A promising future research problem is to investigate the geometric nature of each of these conditions and the potential differences among them.

	\vspace{0.5cm}
	
	{\sc Terence James Gaffney
		
		\vspace{0.5cm}
		
		{\small Department of Mathematics 
			
			Northeastern University
			
			{\tiny 567 Lake Hall - 02115 - Boston - MA, USA, t.gaffney@northeastern.edu }}}
	
	\vspace{0.5cm}	
	
	{\sc Thiago Filipe da Silva
		
		\vspace{0.5cm}
		
		{\small Department of Mathematics
			
			Federal University of Esp\'irito Santo
			
			{\tiny Av. Fernando Ferrari, 514 - Goiabeiras, 29075-910 - Vit\'oria - ES, Brazil, thiago.silva@ufes.br}}}
	

\begin{thebibliography}{99}
		
		\bibitem{FR2} {\sc A. Fernandes and M. A. S. Ruas}, \textit{Bilipschitz determinacy of quasihomogeneous germs}, Glasgow Mathematical Journal, 46 (1), 77-82 (2004).
		
		\bibitem{FR} {\sc A. Fernandes and M. A. S. Ruas}, \textit{Rigidity of bi-Lipschitz equivalence of weighted homogeneous function-germs in the plane}, Proc. Amer. Math. Soc. 141, 1125-1133 (2013).
		
		\bibitem{G1} T. Gaffney, {\it The genericity of the infinitesimal Lipschitz condition for hypersurfaces}, J. Singul. 10 (2014), 108-123. 
		
		\bibitem{SG} {\sc T. Gaffney and T. da Silva}, \textit{Infinitesimal Lipschitz conditions on a family of analytic varieties}, arXiv: 1902.03194 [math.AG]
		
		
		\bibitem{G2} T. Gaffney, {\it Bi-Lipschitz equivalence, integral closure and invariants,} Proceedings of the 10th International Workshop on Real and Complex Singularities. Edited by: M. Manoel, Universidade de S\~ao Paulo, M. C. Romero Fuster, Universitat de Valencia, Spain, C. T. C. Wall, University of Liverpool, London Mathematical Society Lecture Note Series (No.380) November 2010.
		
		\bibitem{G3} T. Gaffney, {\it Integral Closure of Modules and Whitney equisingularity,} Invent. Math. 107 (1992) 301-322.
		
		\bibitem{G4} {\sc T. Gaffney}, {\it Infinitesimal Bi-Lipschitz Equivalence of Functions}, arxiv: 1601.05147v1 [math.CV]  
		
		\bibitem{G5} T. Gaffney, \textit{Generalized Buchsbaum-Rim Multiplicities and a Theorem of Rees}, Communications in Algebra, 31, 3811-3828 (2003).
		
		\bibitem{G6} T. Gaffney, \textit{Equisingularity of Plane Sections, $t^1$ Condition, and the Integral Closure of Modules}, Real and Complex Singularities, ed. by W. L. Marar, Pitman Research Series in Math. 333, Longman, (1995).
		
		\bibitem{G7} T. Gaffney, \textit{The Multiplicity Polar Theorem}, arxiv:math/0703650v1 [math.CV], (2007).
		
		\bibitem{GK} T. Gaffney and S. Kleiman, {\it Specialization of integral dependence for modules,}, Inventiones mathematicae. 137, 541-574 (1999).
		
		\bibitem{KT} S. Kleiman and A. Thorup, {\it A geometric theory of the Buchsbaum-Rim \newline multiplicity}, J. Algebra 167 (1994), 168-231
		
		\bibitem{LT} M. Lejeune-Jalabert and B. Teissier, {\it Cl\^oture int\'egrale des id\'eaux et \newline equisingularit\'e}, S\'eminaire Lejeune-Teissier, Centre de Math\'ematiques \'Ecole Polytechnique, (1974) Publ. Inst. Fourier St. Martin d'Heres, F-38402 (1975).
		
		\bibitem{PT} F. Pham and B. Teissier, {\it Fractions lipschitziennes d'une alg\'ebre analytique complexe et saturation de Zariski}, Centre de Math\'ematiques de l'Ecole Polytechnique (Paris), http://people.math.jussieu.fr/teissier/old-papers.html, June 1969.
		
		\bibitem{L} J. Lipman, {\it Relative Lipschitz-saturation}, Amer. J. Math. 97 (1975), no. 3, 791-813.
		
		\bibitem{L} Looijenga, E. J. N.: {\it Isolated singular points on complete intersection} (Lond. Math. Soc. Lect. Note Ser., vol 77) Cambridge: Cambridge University Press 1984
		
		\bibitem{Z} O. Zariski, {\it General Theory of Saturation and of Saturated Local Rings. II. Saturated local rings of dimension 1} Amer. J. Math. 93 (1971), 872-964
		
		\bibitem{P1} F. Pham, {\it Fractions lipschitziennes et saturation de Zariski des alg\`ebres analytiques complexes.} Expos\'e d'un travail fait avec Bernard Teissier. Fractions lipschitziennes d'une alg\`ebre analytique complexe et saturation de Zariski, Centre Math. l\'Ecole Polytech., Paris, 1969. Actes du Congr\`es International des Math\'ematiciens (Nice, 1970), Tome 2, pp. 649-654. Gauthier-Villars, Paris, 1971.
		
		\bibitem{M1} T. Mostowski, {\it A criterion for Lipschitz equisingularity.} Bull. Polish Acad. Sci. Math. 37 (1989), no. 1-6, 109-116 (1990).
		
		\bibitem{M2} T. Mostowski, {\it Tangent cones and Lipschitz stratifications.} Singularities (Warsaw, 1985), 303-322, Banach Center Publ., 20, PWN, Warsaw, 1988.
		
		\bibitem{PA1} A. Parusinski, {\it Lipschitz stratification of real analytic sets.} Singularities (Warsaw, 1985), 323-333, Banach Center Publ., 20, PWN, Warsaw, 1988. 
		
		\bibitem{PA2} A. Parusinski, {\it Lipschitz properties of semi-analytic sets.} Ann. Inst. Fourier (Grenoble) 38 (1988), no. 4, 189-213.
		
		\bibitem{B} L. Birbrair, {\it Local bi-Lipschitz classification of 2-dimensional semialgebraic sets.} Houston J. Math. 25 (1999), no. 3, 453-472.
		
	\end{thebibliography}
\end{document}